\documentclass[14pt]{amsart}
\usepackage[english]{babel}
\usepackage{amsmath,amsfonts,amssymb,amsthm,mathrsfs,bbm}
\usepackage{graphicx,graphics}
\usepackage{epsfig}
\usepackage{latexsym}
\usepackage[applemac]{inputenc}
\usepackage{ae,aecompl}
\usepackage{pstricks}
\usepackage{enumerate}
\usepackage{xcolor}
\usepackage[pdfpagemode=UseNone,bookmarksopen=false,colorlinks=true,urlcolor=blue,citecolor=blue,citebordercolor=blue,linkcolor=blue]{hyperref}

\usepackage[normalem]{ulem}

\usepackage[top=2.5cm, bottom=2.5cm, left=2.5cm,right=2.5cm]{geometry}
\usepackage{comment}

\usepackage[all, knot, cmtip]{xy}

\newtheorem{theorem}{Theorem}[section]

\newtheorem{lemma}[theorem]{Lemma}

\newtheorem{problem}[theorem]{Problem}

\newcommand{\convdis}{\,{\buildrel d \over \longrightarrow}\,}

\renewcommand{\Pr}[1]{\mathbb{P}(#1)}

\newcommand{\Ex}[1]{\mathbb{E}[#1]}
\newcommand{\Exb}[1]{\mathbb{E}\left[ #1 \right]}

\newcommand{\eqdist}{\,{\buildrel d \over =}\,}

\newcommand{\UIPM}{{\mathbf{m}_\infty}}
\newcommand{\mn}{\mathbf{m}_n}
\newcommand{\m}{\mathbf{m}}

\title[Pattern occurrences in random planar maps]{Pattern occurrences in random planar maps}
\date{}

\author{Michael Drmota}
\address[Michael Drmota]{Institute of Discrete Mathematics and Geometry, TU Wien}
\email{michael.drmota@tuwien.ac.at}

\thanks{The first author is supported by the Austrian Science Fund FWF, Project F50-02 that is part of the SFB ``Algorithmic and Enumerative Combinatorics''. The second author is supported by the Swiss National Science Foundation grant number 200020\_172515.}
\author{Benedikt Stufler}
\address[Benedikt Stufler]{Institute of Mathematics, University of Zurich}
\email{benedikt.stufler@math.uzh.ch}

\begin{document}
	
	\maketitle

\begin{abstract}
We consider planar maps adjusted with a (regular critical) Boltzmann distribution and show that
the expected number of pattern occurrences of a given map is asymptotically linear
when the number $n$ of edges goes to infinity. 
The main ingredient for the proof is an extension of a formula by Liskovets (1999).
\end{abstract}
	

	
\section{Introduction}	

A planar map $m$ is a connected planar
graph,  possibly  with  loops  and  multiple  edges,  together  with  an  embedding  into  the
plane. Usually one edge is directed and distinguished as the {\it root edge}.
There are several ways of introducing a probability distribution on planar maps (see \cite[Ch. 5]{MR3445851}).
In this paper we will focus on Boltzmann weights
\[
W_{\bf q}(m) = \prod_{f\in \mathcal{F}_m} q_{{\rm deg}(f)},
\]
where ${\bf q} = (q_n)_{n\ge 1}$ is a sequence of non-negative weights, $\mathcal{F}_m$ denotes the
set of faces of $m$, and ${\rm deg}(f)$ the degree of a face $f$. Now for every $n$ these weights 
induce a probability distribution on planar maps with $n$ edges. For example, if $q_n = 1$ for all $n$ then
we obtain uniform random maps, whereas if $q_4 =1$ and $q_n=0$ for $n \ne 4$ then we are just considering uniform
quadrangulations. In what follows we will always assume that ${\bf q}$ is regular
critical (see \cite[Sec. 2.1]{Stephenson2016} for a precise definition). This encompasses the case of uniform planar maps and uniform $p$-angulations for all $p \ge 3$ \cite[Sec. 6.1]{Stephenson2016}.

We denote by $\mn$ a random planar map (with $n$ edges) in this sense. Let $\hat m$ denote a plane representation of some planar map. We say that $\hat m$ occurs in ${m}$ as a pattern, 
if $\hat m$ may be embedded in a face-preserving way into the plane representation of ${m}$. 
For example, if $\hat m$ is a cycle of length $d$ then these embeddings correspond to the 
inner faces of $m$ that are cycles and have degree $d$. 

The main purpose of this paper is to get some knowledge of the random number $s(\hat m,\mn)$ of occurrences of $\hat m$ as 
a pattern in $\mn$. 

\begin{theorem}
	\label{te:main2}
	Let $\hat m$ denote a plane representation of some planar map. Then 
	\[
	n^{-1}\Exb{{s(\hat m,\mn)}} \to   \gamma(\hat m) 
	\]
	for some constant $\gamma(\hat m)$ given in Equation~\eqref{eq:const} below. 
\end{theorem}



The main ingredient of the proof is an extension of a formula by 
Liskovets~\cite[Eq. (2.3.1)]{MR1666953} that we obtain using a local limit result by
Stephenson \cite{Stephenson2016}. We will discuss this extension and prove Theorem~\ref{te:main2} in Section~\ref{sec:Benj}. Furthermore, we collect some open problems in  Section~\ref{sec:Prob}.

\subsection*{Notation and Terminology}
 Rooting a planar map at a corner is equivalent to specifying and orienting an edge. We say the origin of this edge is the root vertex of the map. We use the convention that the face to the ``right'' of the oriented root edge is the outer face and is drawn as the unique unbounded face in plane representations. The radius of the map is the maximal distance of a vertex from the root vertex. Let $m$ denote a corner-rooted planar map. We let $v(m)$ and $e(m)$ denote its number of vertices and edges, and $d(m)$ the degree of its root vertex.
For any integer $r \ge 1$ we let  $U_r^c(m)$ denote the corner-rooted submap induced by all vertices with distance at most $r$ from the vertex incident to the root-corner. We let $m^v$ denote the vertex-rooted version of $m$, where we forget about the root-corner and only mark the root-vertex. In particular,  $U_r^v(m) := (U_r^c(m))^v$ is the vertex-rooted version of the neighbourhood $U_r^c(m)$. Maps may be re-rooted. Given a vertex $v$ and a corner $c$ of $m$, we let $(m, v)$ and $(m, c)$ represent the result of re-rooting $m$ at this vertex or corner and forgetting about the original root.

\section{The vicinity of uniformly selected vertices in random maps}\label{sec:Benj}

Liskovets~\cite[Eq. (2.3.1)]{MR1666953} observed that for certain general models of random rooted maps with $n$ edges, the limiting distributions $(d_k)_{k \ge 1}$ for the degree of the root vertex is related to the limiting distribution $(p_k)_{k \ge 1}$ for the degree of a uniformly selected vertex via the formula
\begin{align}
	\label{eq:first}
	p_k = \frac{2}{k \mu }	d_k
\end{align}
for a certain constant $\mu>0$. In the special case of uniform planar maps, the constant equals $\mu= 1/2$ \cite[Prop. 2.6]{MR1666953}. See also further studies of the vertex degrees in models of random planar maps~\cite{MR3071845,llplanar,vertaofa,MR1666953}. 

The regular critical Boltzmann planar map $\mn$ is known to admit a local limit $\UIPM$ by a result of Stephenson~\cite[Thm. 6.1]{Stephenson2016} (see also \cite{MR2013797,2005math.....12304K, MR3183575, MR3083919, MR3256879}). The random infinite map $\UIPM$ describes the  asymptotic behaviour near the root-edge of $\mn$ as $n$ tends to infinity. Using this probabilistic limit, we extend Liskovets' result  by constructing a limit $\m_\infty^*$ that follows a different distribution than $\m_\infty$ and describes the asymptotic vicinity of a typical vertex.


\begin{theorem}
	\label{te:main1}
	The random map $\mn$ rerooted at a uniformly selected vertex admits a distributional limit $\m_\infty^*$ in the local topology. The convergence preserves the embedding in the plane. For any vertex-rooted planar map $\hat{m}$ and any integer $r \ge 1$ it holds that
\begin{align}
\label{eq:second}
\Pr{U_r^v(\m_\infty^*) = \hat{m}} =  \frac{2}{d(m)\mu_v}  \Pr{U_r^v(\UIPM) = \hat m}
\end{align}
for the constant $\mu_v = \lim_{n \to \infty} \Ex{v(\mn)} / n$.
\end{theorem}

Equation~\eqref{eq:second} may be interpreted as an extension of Formula~\eqref{eq:first}. A similar result for the special case of random triangulations and quadrangulations may be obtained by adapting arguments from~\cite{MR2013797,2005math.....12304K}.\footnote{We thank Omer Angel for bringing this to our attention.} Having the limit $\m_\infty^*$ for random quadrangulations at hand, it is possible to use the Tutte bijection to transfer this convergence to the special case of uniform planar maps. However, this case is also encompassed by Theorem~\ref{te:main1} and the approach taken in the present work appears to be simpler and more universal.


The proof of Theorem~\ref{te:main1} requires us to verify the following stochastic re-rooting invariance first. Details on the enumerative background of planar maps may be found in~\cite{MR642058}.
\begin{lemma}
	\label{le:reroot}
	The map $\mn$ is stochastically invariant under re-rooting at a uniformly selected corner.
\end{lemma}
\begin{proof}
	Let $\tilde M$ be  an arbitrary unrooted map with $n$ edges. We have to show that any corner-rooted version of $\tilde M$ corresponds to the same number of choices among the $2n$ corners of $\tilde M$. If this holds, then re-rooting $\mn$ at a uniformly selected corner is identically distributed to $\mn$.
	
	To this end, let us label the corners of $\tilde M$ with numbers from $1$ to $2n$  to form a corner-labelled unrooted map $M$. There are many ways to do this, and we pick an arbitrary one. A permutation $\sigma$ of $\{1, \ldots, 2n\}$ is termed an automorphism of $M$ if the result of relabelling $M$ according to $\sigma$ is identical to $M$. The collection of automorphisms of $M$ is its automorphism group.
	
	If rooting the map $M$ at $1\le i \le 2n$ or $1 \le j \le 2n$ and forgetting about the labels yields two identical unlabelled corner-rooted maps, then there must be an automorphism $\sigma$ of $M$ such that $\sigma(i) = \sigma(j)$. Conversely, if there exists an automorphism $\sigma$ with $\sigma(i) = j$, then clearly rooting $M$ at $i$ or $j$ yields identical unlabelled corner-rooted maps. Moreover, if $\sigma$ and $\nu$ are automorphisms of $M$ that both satisfy $\sigma(i) = j$ and $\nu(i) = j$, then $\sigma \nu^{-1}$ is an automorphism that fixes the label $i$. Hence $\sigma \nu^{-1}$ is an automorphism of a corner-rooted labelled planar map. Corner-rooted maps are asymmetric, so $\sigma \nu^{-1}$ must be the identity permutation, that is, $\sigma = \nu$. Thus, the cardinality of the automorphism group of $M$ is equal to the number of corners (among the $n$ choices) such that rooting $M$ at this corner yields the same unlabelled corner-rooted map as 
	rooting $M$ at the corner $i$. This number does not depend on $i$, so to any corner-rooted version of $\m$ corresponds to the same number of choices for root-corners.  
\end{proof}

\begin{proof}[Proof of Theorem~\ref{te:main1}]
We let $v_n$ denote a uniformly at random selected vertex of the Boltzmann map $\mn$. Let $m$ be a fixed corner-rooted planar map. Let $\kappa(m)$ denote the number of corners $c$ incident to the root-vertex of $m$ with the property that $m$ is invariant under re-rooting at $c$.  For any integer $r \ge 1$ we may write
\begin{align*}
\Pr{U_r^v(\mn, v_n) = m^v} &= \Ex{X_n / v(\mn)},
\end{align*}
with $X_n$ denoting the number of vertices $v$ in $\mn$ such that $U_r^v(\mn, v) = m^v$.
 To each such vertex correspond precisely $\kappa(m)$ corners $c$ with $U_r^c(\mn,c) = m$. Thus the total number $Y_n$ of corners whose corner-rooted $r$-neighbourhood equals $m$ satisfies
\[
Y_n = \kappa(m) X_n.
\]
Hence
\begin{align}
\label{eq:prob}
\Pr{U_r^v(\mn, v_n) = m^v} =  \Exb{\frac{ Y_n }{v(\mn) \kappa(m)}}.
\end{align}
The map $\mn$ is stochastically invariant under re-rooting at a uniformly select corner $c_n$. Thus
\begin{align*}
\Ex{ Y_n /(2n)}  &= \Pr{U_r^c(\mn,c_n) = m} \\
&=  \Pr{U_r^c(\mn) = m} \\
&\to \Pr{U_r^c(\UIPM) = m}.
\end{align*}
The large deviation bounds~\cite[Lem. 6.6]{Stephenson2016} imply that there is a constant $\mu_v > 0$ and a sequence $t_n$ with $t_n \to 0$ such that
\begin{align}
\label{eq:concentration}
|v(\m_n) / n - \mu_v| \le t_n
\end{align} 
holds with probability tending to $1$ as $n$ tends to infinity. (In the case of uniform maps the fluctuation may even be precisely quantified by the normal distribution, see Lemma~\ref{le:cltvert}. For our purposes the concentration result~\eqref{eq:concentration} in a more general setting suffices.) Using $Y_n / (v(\mn) \kappa(m)) \le 1$  it follows that
\begin{align*}
\Exb{\frac{ Y_n }{v(\mn) \kappa(m)}} &= o(1) + \Exb{\frac{ Y_n }{v(\mn) \kappa(m)}, |v(\mn)/n - \mu_v| \le t_n} \\
&\le o(1) + \Exb{\frac{ Y_n }{n(\mu_v - t_n) \kappa(m)}}\\
&= o(1) + \frac{2}{\kappa(m)\mu_v}\Pr{U_r^c(\UIPM) = m}.
\end{align*}
Similarly, we obtain a lower bound, as $Y_n \le 2n$ implies that
\begin{align*}
\Exb{\frac{ Y_n }{v(\mn) \kappa(m)}} &\ge o(1) + \Exb{\frac{ Y_n }{n(\mu_v + t_n) \kappa(m)}, |v(\mn)/n - \mu_v| \le t_n} \\
&= o(1) + \Exb{\frac{ Y_n }{n(\mu_v + t_n) \kappa(m)}}.
\end{align*}
By Equation~\eqref{eq:prob} this implies
\begin{align*}
\Pr{U_r^v(\mn, v_n) = m^v} &\to \frac{2}{\kappa(m)\mu_v} \Pr{U_r^c(\UIPM) = m} \nonumber \\
&= \frac{2}{\alpha(m)\kappa(m)\mu_v}  \Pr{U_r^v(\UIPM) = m^v}
\end{align*}
with $\alpha(m)$ denoting the number of different corner-rooted maps that may be obtained by re-rooting $m$ at a corner incident to the root-vertex. By standard properties of group operations it holds that
\[
	\alpha(m) \kappa(m) = d(m).
\] Thus
\begin{align}
\label{eq:rel}
\lim_{n \to \infty}
\Pr{U_r^v(\mn, v_n) = m^v} =  \frac{2}{d(m)\mu_v}  \Pr{U_r^v(\UIPM) = m^v}
\end{align}

We are now going to show that this implies distributional convergence for the neighbourhood $U_r^v(\mn, v_n)$. For any vertex rooted planar map $\hat m$ let us set \[p_{\hat m, n} := \Pr{U_r^v(\mn, v_n) = \hat m}\] and \[p_{\hat m} := \frac{2}{d(\hat m) \mu_v} \Pr{U_r^v(\UIPM) = \hat m}.\]
In order to deduce weak convergence of $U_r^v(\mn, v_n)$ it remains to verify
\begin{align}
\label{eq:toverify}
\sum_{\hat m} p_{\hat m} = 1
\end{align}
with the sum index ranging over all vertex-rooted planar maps $\hat m$. To this end, let $X_n(\hat m)$ denote the number of vertices $v$ in $\mn$ with $U_r^v(\mn,v) = \hat m$. Let $Y_n(\hat m)$ denote the number of corners in $\mn$ whose vertex-rooted $r$-neighbourhood equals~$\hat m$. For any fixed $K \ge 1$ it follows from Inequality~\eqref{eq:concentration} that
\begin{align*}
\sum_{k \ge K} \sum_{\hat m, e(\hat m)=k} p_{\hat m,n} &= \Exb{ \sum_{k \ge K} \sum_{\hat m, e(\hat m)=k} \frac{X_n(\hat m)}{v(\mn)} } \\
&\le o(1) + C \Exb{\sum_{k \ge K} \sum_{\hat m, e(\hat m)=k} \frac{Y_n(\hat m)}{2n}} \\
&= o(1) + C \Pr{ e( U_r(\mn)) \ge K}
\end{align*}
for some bound $C>0$ that does not depend on $n$ (or $k$ or $\hat m$) and an $o(1)$ term that converges to zero uniformly in $k$ and $\hat m$ as $n$ becomes large. Since $ U_r^c(\mn) \convdis U_r^c(\UIPM)$ it follows that for any $\epsilon>0$ we may select $K\ge 1$ large enough such that 
\[\sum_{k \ge K} \sum_{\hat m, e(\hat m)=k} p_{\hat m,n} < \epsilon
\]
for large enough $n$. This entails
\[
	\sum_{k < K} \sum_{\hat m, e(\hat m) = k} p_{\hat m} \ge 1 - \epsilon.
\]
We have thus proved Equation~\eqref{eq:toverify}. Hence there is a random vertex-rooted planar map $V_r$ with distribution $\Pr{V_r = \hat m} = p_{\hat m}$ such that
\begin{align}
	\label{eq:conv}
	U_r^v(\mn, v_n) \convdis V_r
\end{align}
as $n$ becomes large. The family $(V_r)_{r \ge 1}$ forms a projective system with respect to the projections $U_r^v(\cdot)$, since for any $1 \le s \le r \le t$ it holds that
\[
	U_s^v(U_r^v(V_t)) \eqdist U_s^v(V_t).
\]
It follows by a general result \cite[Ch. 9, \S 4, No. 3, Theorem 2]{MR0276436}  that there is a random infinite planar map $\m_\infty^*$ such that
\[
	U_r^v(\m_\infty^*) \eqdist V_r
\]
for all $r \ge 1$. By \eqref{eq:conv} it follows that $\m_\infty^*$ is the distributional limit of the random planar map $\m_n$ rerooted at a uniformly selected vertex, and the convergence preserves the embedding in the plane.
\end{proof}

We are now ready to prove our main result.

\begin{proof}[Proof of Theorem~\ref{te:main2}]
	Let $\m$ be a fixed corner-rooted version of the plane map $\hat m$ such that the plane representation of $\m$ that has the unbounded face to the right of the root-edge coincides with  $\hat m$. We say $\m$ occurs as a pattern at a corner $c$ of $\mn$ if $\m$ may be embedded into $\m$ in a face-preserving way such that the root-corner of $\m$ gets mapped to the corner $c$.
	
	If we count the number $Z_n$ of corners of $\mn$ where $\m$ appears as a pattern, then we over-count the occurrences of the unrooted plane map $\hat m$. If $\beta(\hat m)$ denotes the number of ways that $\hat m$ may be rooted at an half-edge of its boundary, then
	\[
	s(\hat m, \mn) = Z_n / \beta(\hat m).
	\]
	Lemma~\ref{le:reroot} together with the convergence of $\mn$ towards the UIPM $\UIPM$ implies that $\Ex{Z_n /2n}$ converges to the probability $q(\hat m)$ that $\hat m$ occurs as a pattern at the root of $\UIPM$. Hence
	\begin{align}
	\label{eq:const}
	\Exb{\frac{s(\hat m,\mn)}{n}} \to \frac{2 q(\hat m)}{\beta(\hat m)}.
	\end{align}
\end{proof}

\section{Open Problems}\label{sec:Prob}

In Theorem~\ref{te:main1} we have shown that there is a Benjamini--Schramm limit ${\bf m}_\infty^*$ of random planar maps.
However, this limit graph has no explicit description. In particular it is not clear how the probability distribution 
of some (simple) parameters of ${\bf m}_\infty^*$ can be computed. 
For example, it would be nice to have a proper representation of the constant $\gamma(\hat m)$ in Theorem~\ref{te:main2}.
\begin{problem}
Describe the Benjamini--Schramm limit ${\bf m}_\infty^*$ of random planar maps in a proper (explicit) way.
\end{problem}

Another open question is to make the pattern count asymptotics of Theorem~\ref{te:main2} more precise. 
Actually a  central limit theorem is expected (as given, for example, in \cite{MR2095934}
for random quadrangulations and 2-connected triangulations or in \cite{MR3071845} for vertices of 
degree $k$ in random maps or 2-connected maps).

\begin{problem}
Does the number of occurrences $s(\hat m, \mn)$ of a pattern $\hat m$ in a random planar map $\mn$
satisfy a central limit theorem (similarly to Lemma~\ref{le2})?
\end{problem}


\section{Appendix}


We have used in the proof of Theorem~\ref{te:main1} that $v(\mn)/n$ is close to constant 
with high probability. For the case of uniform planar maps we make this more precise.
The following central limit theorem seems to be {\it classical} in the theory of random planar maps, however, 
the only explicit reference we found is a lecture by Marc Noy at the Alea-meeting 2010 in Luminy.\footnote
{{\tt https://www-apr.lip6.fr/alea2010/}} We give a proof that is based on the {\it Quadratic Method}.

\begin{lemma}\label{le2}
	\label{le:cltvert} Let $\mn$ denote the uniform planar map. 
	The number of vertices $v(\mn)$ satisfies a central limit of the form
	\[
	\frac{v(\mn) - n/2}{\sqrt{25 n/32}} \convdis \mathcal{N}(0, 1)
	\]
	with $\Exb{v(\mn)} = \frac n2 +1$ and $\mathbb{V}{\rm ar}[\mn]  = 25n/32+ O(1)$.
\end{lemma}
\begin{proof}
	Let $M(z,x,u)$ denote the generating function of rooted planar maps, where the variable $z$ 
	corresponds to the number of edges, $x$ to the number of vertices and $u$ to the root face valency.
	Then by the usual combinatorial decomposition of maps we have
	\[
	M(z,x,u) = x + zu^2 M(z,x,u)^2 + zu \frac{M(z,x,1)- u M(z,x,u)}{1-u}
	\]
	and by the quadratic method we can express $M(z,x,1)$ as a rational function in $u = u(z,x)$
	that is given by the solution of the algebraic equation
	\[
	4u^4xz+u^4z^2-2u^4z-8u^3xz+4u^3z+4u^2xz+2u^3-2u^2z-7u^2+8u-3 = 0,\, u(0,1) = 1,
	\]
	from which we obtain a singular expansion of the form
	\[
	u(z,x) = u_0(x) + u_1(x) \sqrt{1- \frac z{\rho(x)}} + u_2(x) \left( 1- \frac z{\rho(x)} \right) 
	+ u_3(x)\left( 1- \frac z{\rho(x)} \right)^{3/2}+\cdots,
	\]
	where the functions $u_j(x)$ are analytic at $x=1$, satisfy
	$u_0(1) = \frac 65$, $u_1(1) = -\frac 6{25}$, $u_2(1) = \frac 6{125}$, $u_3(1) = -\frac 6{625}$,
	and the function $\rho(x)$ satisfies the equation
	\begin{multline*}
	3072x^3z^4-4608x^2z^4-1536x^2z^3+4608xz^4+1536xz^3\\-1536z^4+192xz^2+768z^3-96z^2 = 0
	\end{multline*}
	with $\rho(1) = \frac 1{12}$. From this it follows that 
	\[
	M(z,x,1) = \frac{1-(4xz-z^2)u^4-(-8x+2)zu^3-(-1+(4x-2)z)u^2-2u}{4(1-u)u^3z^2}
	\]
	(with $u = u(z,x)$) has a local representation of the form
	\[
	M(z,x,1) = b_0(x) + b_2(x)\left( 1- \frac z{\rho(x)} \right) + b_3(x)\left( 1- \frac z{\rho(x)} \right)^{3/2} + \cdots,
	\]
	where the functions $b_j(x)$ are analytic at $x=1$ and satisfy $b_0(1) = \frac 43$, $b_2(1) = -\frac 43$, $b_3(1) = \frac 83$.
	
	At this stage we can apply standard tools (see \cite[Chapter 2]{MR2484382}) 
	to obtain a central limit theorem for $v(\mn)$ of the form
	$(v(\mn) - \mu n)/{\sqrt{\sigma^2 n}} \convdis \mathcal{N}(0, 1)$, where
	\[
	\mu = - \frac{\rho'(1)}{\rho(1)}, \quad \sigma^2 = \mu + \mu^2 - \frac{\rho''(1)}{\rho(1)}.
	\]
	Since $\rho'(1) = - \frac 1{24}$ and $\rho''(1) = - \frac 1{384}$ we immediately obtain
	$\mu = \frac 12$ and $\sigma^2 = \frac {25}{32}$. We also have 
	$\Exb{v(\mn)} = \mu n + O(1)$ and $\mathbb{V}{\rm ar}[\mn]  = \sigma^2 n +  O(1)$.
	In this special case Euler's relation and duality can be used to obtain (the even more precise representation)
	$\Exb{v(\mn)} = n/2 + 1$.
\end{proof}


\bibliographystyle{siam}
\bibliography{maps}

\begin{thebibliography}{10}

\bibitem{MR2013797}
{\sc O.~Angel and O.~Schramm}, {\em Uniform infinite planar triangulations},
  Comm. Math. Phys., 241 (2003), pp.~191--213.

\bibitem{MR3183575}
{\sc J.~E. Bj{\"o}rnberg and S.~{\"O}. Stef{\'a}nsson}, {\em Recurrence of
  bipartite planar maps}, Electron. J. Probab., 19 (2014), pp.~no. 31, 40.

\bibitem{MR3445851}
{\sc M.~B\'{o}na}, {\em Introduction to enumerative and analytic
  combinatorics}, Discrete Mathematics and its Applications (Boca Raton), CRC
  Press, Boca Raton, FL, second~ed., 2016.

\bibitem{MR0276436}
{\sc N.~Bourbaki}, {\em \'{E}l\'ements de math\'ematique. {F}asc. {XXXV}.
  {L}ivre {VI}: {I}nt\'egration. {C}hapitre {IX}: {I}nt\'egration sur les
  espaces topologiques s\'epar\'es}, Actualit\'es Scientifiques et
  Industrielles, No. 1343, Hermann, Paris, 1969.

\bibitem{llplanar}
{\sc G.~Collet, M.~Drmota, and L.~Klausner}, {\em Limit laws of planar maps
  with described vertex degrees}, Combinatorics, Probability and Computing, to
  appear.

\bibitem{vertaofa}
{\sc G.~{Collet}, M.~{Drmota}, and L.~D. {Klausner}}, {\em {Vertex Degrees in
  Planar Maps}}, Proceedings AofA, 2016.

\bibitem{MR3083919}
{\sc N.~Curien, L.~M\'enard, and G.~Miermont}, {\em A view from infinity of the
  uniform infinite planar quadrangulation}, ALEA Lat. Am. J. Probab. Math.
  Stat., 10 (2013), pp.~45--88.

\bibitem{MR2484382}
{\sc M.~Drmota}, {\em Random trees -- An interplay between combinatorics and
  probability}.

\bibitem{MR3071845}
{\sc M.~Drmota and K.~Panagiotou}, {\em A central limit theorem for the number
  of degree-{$k$} vertices in random maps}, Algorithmica, 66 (2013),
  pp.~741--761.

\bibitem{MR2095934}
{\sc Z.~Gao and N.~C. Wormald}, {\em Asymptotic normality determined by high
  moments, and submap counts of random maps}, Probab. Theory Related Fields,
  130 (2004), pp.~368--376.

\bibitem{2005math.....12304K}
{\sc M.~{Krikun}}, {\em {Local structure of random quadrangulations}}, ArXiv
  Mathematics e-prints,  (2005).

\bibitem{MR642058}
{\sc V.~A. Liskovets}, {\em A census of nonisomorphic planar maps}, in
  Algebraic methods in graph theory, {V}ol. {I}, {II} ({S}zeged, 1978), vol.~25
  of Colloq. Math. Soc. J\'anos Bolyai, North-Holland, Amsterdam-New York,
  1981, pp.~479--494.

\bibitem{MR1666953}
{\sc V.~A. Liskovets}, {\em A pattern of asymptotic vertex valency
  distributions in planar maps}, J. Combin. Theory Ser. B, 75 (1999),
  pp.~116--133.

\bibitem{MR3256879}
{\sc L.~M{\'e}nard and P.~Nolin}, {\em Percolation on uniform infinite planar
  maps}, Electron. J. Probab., 19 (2014), pp.~no. 79, 27.

\bibitem{Stephenson2016}
{\sc R.~Stephenson}, {\em Local convergence of large critical multi-type
  galton--watson trees and applications to random maps}, Journal of Theoretical
  Probability,  (2016).

\end{thebibliography}

\end{document}